\documentclass{amsart}
\usepackage{amsmath}
\usepackage[dvips]{graphicx}
\usepackage[table,usenames,dvipsnames]{xcolor}
\usepackage{multirow}
\usepackage{array}
\usepackage{longtable}
\usepackage{booktabs}
\usepackage{listings}
\usepackage{array}

\usepackage{rotating}
\usepackage{pgf,tikz}
\usepackage{subfig}
\usepackage{amsmath, amsthm, amscd, amsfonts, amssymb, color}
\usepackage[bookmarksnumbered, plainpages]{hyperref}
\usepackage{relsize}
\usepackage{longtable}
\usepackage{algorithm}
\usepackage{algorithmicx}
\usepackage{algpseudocode}
\newcolumntype{G}{>{\columncolor{gray!15}}c}
\newcolumntype{W}{c}
\newcommand{\Input}{\State \textbf{Input:} }
\newcommand{\Output}{\State \textbf{Output:} }
\newtheorem{thm}{Theorem}[section]

\newtheorem{con}[thm]{Conjecture}
\newtheorem{lem}[thm]{Lemma}

\newtheorem{exm}[thm]{Example}
\newtheorem{rem}[thm]{Remark}

\numberwithin{equation}{section}
\newcommand\blfootnote[1]{%
  \begingroup
  \renewcommand\thefootnote{}\footnote{#1}%
  \addtocounter{footnote}{-1}%
  \endgroup
}
\begin{document}
\title[Reconstruction of Permutations]{The Sequence Reconstruction of Permutations under Hamming Metric with Small Errors}
\author[\tiny{Abdollahi}]{A. Abdollahi}%
\author[\tiny{Bagherian}]{J. Bagherian}
\author[\tiny{Eskandari}]{H. Eskandari}%
\author[\tiny{Jafari}]{F. Jafari}
\author[\tiny{Khatami}]{M. Khatami}%
\address{Department of Pure Mathematics, Faculty of Mathematics and Statistics,  University of Isfahan, Isfahan 81746-73441, Iran.}%
\email{a.abdollahi@math.ui.ac.ir}%
\email{bagherian@sci.ui.ac.ir}
\email{Hajjar1360Eskandari@gmail.com}%
\email{math\_fateme@yahoo.com}
\email{m.khatami@sci.ui.ac.ir}
\author[\tiny{Parvaresh}]{F. Parvaresh}
\address{ Department of Electrical Engineering, Faculty
of Engineering, University of Isfahan, Isfahan 81746-73441, Iran.}
\email{f.parvaresh@eng.ui.ac.ir}
\author[\tiny{Sobhani}]{R. Sobhani}%
\address{Department of Applied Mathematics and Computer Science, Faculty of Mathematics and Statistics,  University of Isfahan, Isfahan 81746-73441, Iran.}%
\email{r.sobhani@sci.ui.ac.ir}
\thanks{Corresponding Author: A. Abdollahi}
%\thanks{F. Parvaresh is supported  in part by grant No. 1401680050.}
%\subjclass[2020]{94B25; 94B65; 68P30}

\begin{abstract}
The sequence reconstruction problem asks for the recovery of a sequence from multiple noisy copies, where each copy may contain up to $r$ errors. In the case of permutations on \(n\) letters under the Hamming metric, this problem is closely related to the parameter $N(n,r)$, the maximum intersection size of two Hamming balls of radius $r$.

While previous work has resolved \(N(n,r)\) for small radii (\(r \leq 4\)) and established asymptotic bounds for larger \(r\), we present new exact formulas for \(r \in \{5,6,7\}\) using  group action techniques. In addition, we develop a  formula for \(N(n,r)\) based on the irreducible characters of the symmetric group \(S_n\), along with an algorithm that enables computation of \(N(n,r)\) for larger parameters, including cases such as \(N(43,8)\) and \(N(24,14)\).
\end{abstract}
\maketitle
\section{Introduction and Results}

The sequence reconstruction problem, introduced by Levenshtein in 2001 \cite{[14], [15]},
involves reconstructing a transmitted sequence from multiple noisy copies received over distinct channels, each introducing at most $r$ errors. Later, this problem has attracted interest in applications such as  DNA storage \cite{[3],[13]}, racetrack memories \cite{[2]}, and communication systems. In combinatorial terms, a key challenge is to determine the maximum intersection size  $N(n, r)$ of two metric balls of radius $r$ centered at distinct sequences. For permutations under the Hamming metric, this problem remains central due to the relevance of permutation codes in flash memories \cite{[8]}, power-line communications \cite{[18],[22]}, and DNA storage \cite{[1],[19]}.  

Main contributions of \cite{WFK} resolves the sequence reconstruction problem for permutations under Hamming errors for small radii ($r \leq 4$) and provides asymptotic bounds for larger $r$. In particular,  it is proved that $N(n, 2) = 3$ ($n \geq 3$) \cite[Theorem 4]{WFK}; $N(n, 3) = 4n - 6$ ($n \geq 3$) \cite[Theorem 5]{WFK}; and $N(n, 4) = 7n^2 - 31n + 36$ ($n \geq 4$)  \cite[Theorem 6]{WFK}. For example, the latter implies that unique reconstruction requires $N(n,4)+1=7n^2 - 31n + 37$ distinct permutations at distance $\leq 4$.     

Here by using some elementary results about action of groups on sets, we find an algorithm with input $r$ and output $N(n,r)$. Applying the algorithm we find the exact values of $N(n,r)$ for $r\in \{5,6,7\}$, as follows:
\begin{thm}\label{th5}
 The following hold:
\begin{itemize}
\item[(i)]  $ N(n,5) = \dfrac{32}{3}n^3 - 89n^2 + \dfrac{739}{3}n - 220$, for  $ n \geq 5$.
\item[(ii)] $ N(n,6) = \dfrac{181}{12}n^4 - \dfrac{401}{2}n^3 + \dfrac{11783}{12}n^2 - \dfrac{4153}{2}n + 1590$,  for  $ n \geq 6$.
\item[(iii)] $N(n,7)= \dfrac{607}{30}n^5 - \dfrac{4675}{12}n^4 + \dfrac{8807}{3}n^3 - \dfrac{129041}{12}n^2 + \dfrac{190471}{10}n - 12978$, for $ n \geq 7$.
\end{itemize}
\end{thm}
Furthermore, in Theorem  \ref{thchar}, we provide a formula for computing $N(n,r)$ using the irreducible characters of the symmetric group $S_n$ and an algorithm for computing \(N(n,r)\) in this way. The results obtained from implementing this algorithm in \textsf{GAP} \cite{GAP} are presented in Table \ref{table2}, which displays the computed values of $N(n,r)$.
\section{Preliminaries}

 Let $n$ be a positive integer and let $[n]:=\{1,\dots, n\}$. For any non-empty subset $A$ of $[n]$ we denote by $S_A$ the set of all permutations on $[n]$ such that they induce the identity map on $[n]\setminus A$; in particular $S_{[m]}$ is denoted by $S_m$ for any $m\in [n]$. We denote the identity element of $S_n$ by $I_n$. For $\sigma \in S_n$, $M(\sigma):=\{i\in [n] \;|\; i^\sigma\neq i\}$, where $i^\sigma$ denotes the image of $i$ under $\sigma$. We define $T_r^n$  as $\{\sigma \in S_n \;|\; |M(\sigma)|\leq r\}$. The composition $\sigma \tau$ of two permutations  $\sigma,\tau \in S_n$ is defined by $i^{\sigma \tau}=(i^\sigma)^\tau$ for any $i\in [n]$.  It is well-known that every permutation in the  \( S_n \) can be uniquely expressed as a product of disjoint cycles, except for changes in the order of the factors, so that the rule of commutation is the property of being disjoint. The cycle type of a permutation \( \pi \in S_n \) is a sequence \(1^{t_1}, \dots , n^{t_n} \),   where $t_i$ denotes the number of disjoint cycles of length $i$ in the cycle decomposition of \( \pi \). By convention, any term with \(t_i = 0\) is omitted; this applies to \(1^{t_1}\) as well, so fixed points are not shown. Note that for a permutation $\pi$ with such a cyclic type we have $|M(\pi)|=\sum_{i=2}^{n}i\times t_i$.

  For two permutations $\delta$ and $\sigma$, we denote $\delta^{-1} \sigma \delta$ by $\sigma^\delta$  which is called the conjugate of $\sigma$ by $\delta$. This defines an equivalence relation on $S_n$, whose equivalence classes are called conjugacy classes.  
The conjugacy class of a permutation $\sigma \in S_n$ is the set
$
 \{ \sigma^\delta \mid \delta \in S_n \}.
$  It is well known that two permutations in $S_n$ are conjugate if and only if they have the same cycle type; therefore, each cycle type corresponds to a unique conjugacy class of $S_n$.
For any $\pi \in S_n$ and $r,s\in [n]$ such that $r\leq s$, we denote the set $\{(\sigma,\tau)\in T_r^{s} \times T_r^{s} \;|\; \sigma \tau =\pi\}$ by $\Omega_{s,r}(\pi)$. We denote by $C_{S_A}(\sigma)$ the centralizer of $\sigma$ in $S_A$, that is, $\{\delta \in S_A \;|\; \delta \sigma=\sigma \delta\}$.

For any two permutations $\sigma, \pi \in S_n$, the Hamming distance $d(\sigma, \pi)$ between $\sigma$ and $ \pi$  is defined by $|M(\sigma\pi^{-1})|$. For a given $\sigma \in S_n$ and a non-negative integer $r$, the Hamming ball of radius $r$ centered at $\sigma$ is defined as
$
B_r^n(\sigma) := \{\, \pi \in S_n \;|\; d(\sigma,\pi) \le r \,\}.$
For integers $d$ and $r$, define $I(n,d,r)$ as the maximum possible size of the intersection of two metric balls of radius $r$ in $S_n$, whose centers are exactly at distance $d$:
\[
I(n,d,r) := \max \bigl\{ |B_r^n(\pi) \cap B_r^n(\tau)| \; : \; \pi, \tau \in S_n, \; d(\pi,\tau) = d \bigr\}.
\]

Then  $N(n,r)$, which represents the minimum number of channels needed to guarantee correct decoding of any transmitted permutation with up to $r$ errors per channel, can be expressed as
\[
N(n,r) := \max_{\pi, \tau \in S_n,\, d(\pi,\tau)\ge 2} |B_r^n(\pi) \cap B_r^n(\tau)| 
= \max_{d \ge 2} I(n,d,r).
\]

The authors in~\cite{WFK} determined the exact values of $I(n,d,r)$ for $d = 2,3,4$ and for any $r \ge 2$. Moreover, they proposed the following conjecture:
\begin{con}\cite{WFK}\label{con}
For any $r \ge 5$ and $n \ge r$, we have
$
N(n,r) = I(n,2,r).
$
\end{con}

In \cite{WFK}, it is also shown that  Conjecture \ref{con} is valid for $r = 3, 4$ and for any $r \ge 5$ when $n$ is large enough \cite[Theorem 9]{WFK} and \cite[Lemma 3]{WFK}. In this paper, we will prove the validity of this conjecture for $r = 5, 6, $ and $7$. Most of the following lemma, which will be used in later sections, has been proved in  \cite[Lemma 7]{WK}. It is presented here for the reader’s convenience and to clarify its relation to the new notations.
\begin{lem}\label{lme6}
For any $r\in [n]$ and for all $\pi,\sigma \in S_n$, we have $|\Omega_{n,r}(\pi)|=|B_r^n(\pi)\cap B_r^n(I_n)|$ and 
 $|\Omega_{n,r}(\pi^\sigma)|=|\Omega_{n,r}(\pi)|=|T_r^n \pi \cap T_r^n|$, where $T_r^n \pi:=\{\rho\pi\,:\, \rho \in T_r^n\}$.
In particular, if $\pi \in T_k^n$ for some $k\in [n]$, then $|\Omega_{n,r}(\pi)|=|\Omega_{n,r}(\pi')|$ for some $\pi' \in S_k$. 
\end{lem}
\begin{proof}
The map defined from $\Omega_{n,r}(\pi)$ to $\Omega_{n,r}(\pi^\sigma)$ by $(\tau_1,\tau_2)\mapsto (\tau_1^\sigma,\tau_2^\sigma)$ is a bijection. The map defined from 
$T_r^n \pi \cap T_r^n$ to $\Omega_{n,r}(\pi)$ by $\tau \mapsto (\pi \tau^{-1},\tau)$ is a bijection, to see the latter we need the fact that $M(\theta)=M(\theta^{-1})$ for all $\theta \in S_n$. The proof of the first part is complete with the fact that $B_r^n(\sigma)=T_r^n \sigma$, for all $\sigma \in S_n$. 

The second part follows from the first equality of the first part and the fact that one can find a $\theta \in S_n$ such that $\pi^\theta \in S_k$ whenever $|M(\pi)|\leq k$. This completes the proof.
\end{proof}
\begin{rem}\label{rem}
Since for any permutation $\sigma$ with $|M(\sigma)| > 2r$ we have $\Omega_{n,r}(\sigma) = \varnothing$, and in view of \cite[Lemma 3]{WFK} and Lemma $\ref{lme6} $, it follows that if $\sigma_1,\dots,\sigma_\ell$ are representatives of those conjugacy classes in $S_n$ whose elements move at most $2r$ points, then 
$
N(n,r) = \max \bigl\{\, |\Omega_{n,r}(\sigma_i)| : 1 \leq i \leq \ell \,\bigr\}.
$ Moreover, for a fixed integer $d$ with $0 \leq d \leq 2r$, if $\sigma'_1,\dots,\sigma'_k$ are representatives of those conjugacy classes in $S_n$ whose elements move exactly $d$ points, then
$
I(n,d,r) = \max \bigl\{\, |\Omega_{n,r}(\sigma'_i)| : 1 \leq i \leq k \,\bigr\}$. In particular, $I(n,2,r) = |\Omega_{n,r}(\tau)|$, where $\tau$ is a transposition in $S_n$, i.e., a permutation that interchanges two elements and fixes all others.
\end{rem}
We will need the concept of group action in the next section, and we recall it here. 
A right action of a group $G$ on a set $X$ is a map
$
X \times G \to X, \quad (x,g) \mapsto x \cdot g,
$
such that $(x \cdot g)\cdot h = x \cdot (gh)$ for all $g,h \in G$ and $x \cdot I_n = x$ for all $x \in X$. 
For $x \in X$, the orbit of $x$ is
$
\operatorname{Orb}(x) = \{x \cdot g : g \in G\},
$
and the  stabilizer of $x$ is
$
\operatorname{Stab}_G(x) = \{g \in G : x \cdot g = x\}.
$
%It is well known that
%$
%|G| = |\operatorname{Orb}(x)| \cdot |\operatorname{Stab}(x)|.
%$
\section{Exact values $N(n,r)$ for $r = 5,6,7$}
In this section,  for a given integer $r\geq 2$, we prove several lemmas showing that for any $n \ge 2r$ and any $\sigma \in S_n$ with $|M(\sigma)| \le 2r$, the computation of $|\Omega_{n,r}(\sigma)|$ in $S_n$ can be reduced to the corresponding computation in $S_{2r}$. We then present an algorithm for determining $N(n,r)$, and use this algorithm to prove Theorem~\ref{th5}.
\begin{lem}\label{lem1}
Let $r\in [n]$ such that $2r\leq n$ and $\sigma \tau=\pi$, where $\sigma, \tau \in T_r^n$ and $\pi \in S_{2r}$. Then there exist $\sigma',\tau'\in T_r^{2r}$ such that $\sigma' \tau'=\pi$, where $\sigma'=\sigma^\delta$, $\tau'=\tau^\delta$ for some $\delta\in C_{S_n}(\pi)$. 
\end{lem}
\begin{proof}
Let $\Delta := (M(\sigma)\cup M(\tau))\cap [2r]$ and $\Delta' := (M(\sigma)\cup M(\tau))\setminus \Delta$. Since $\sigma,\tau\in T_r^n$, we have $|M(\sigma)\cup M(\tau)|\le 2r$, and hence $f:=|\Delta'|\le 2r-|\Delta|$. Note that if $f=0$, then there is nothing to prove; hence, we may assume that $f>0$. Choose a subset $\Lambda\subseteq [2r]\setminus \Delta$ such that $|\Lambda|=f$. Choose a permutation $\delta\in S_n$ such that $\delta(x)\in\Lambda$ for all $x\in\Delta'$, $\delta(x)\in\Delta'$ for all $x\in\Lambda$, and $x^\delta=x$ for all $x\in [n]\setminus(\Delta'\cup\Lambda)$. Such a permutation exists since $|\Delta'|=|\Lambda|$. Conjugating the equality $\sigma\tau=\pi$ by $\delta$, we obtain $\sigma^\delta\tau^\delta=\pi^\delta$. Since $M(\pi)\subseteq \Delta$ and $\delta$ acts trivially on $\Delta$, it follows that $\pi^\delta=\pi$, and hence $\delta\in C_{S_n}(\pi)$. By construction, all points in $M(\sigma^\delta)$ and $M(\tau^\delta)$ lie in $[2r]$. Moreover, since $|M(\sigma^\delta)|=|M(\sigma)|$ and $|M(\tau^\delta)|=|M(\tau)|$, $\sigma^\delta,\tau^\delta\in T_r^{2r}$. Setting $\sigma'=\sigma^\delta$ and $\tau'=\tau^\delta$ completes the proof.
\end{proof}
The following example demonstrates the construction described in Lemma~\ref{lem1}.
\begin{exm}
Let $n=12$ and $r=5$. Consider the permutation $\pi=(1\ 3\ 8\ 6)\in S_{10}$. Let $\sigma=(1\ 3\ 8\ 11\ 12)$ and $\tau =(12\ 11\ 6\ 1)$ in $S_{12}$. Clearly,   $\sigma\tau =\pi$, $\Delta := (M(\sigma)\cup M(\tau))\cap [2r]=\{1,3,6,8\}$ and $\Delta' := (M(\sigma)\cup M(\tau))\setminus \Delta=\{11,12\}$. We choose $\Lambda=\{2,4\}$ and define $\delta=(11\ 2\ 12\ 4)\in S_{12}$ so that the conditions on $\Lambda$ and $\delta$ stated in the proof of Lemma~\ref{lem1} are satisfied. It is easy to see that $\delta\in C_{12}(\pi)$ and if   $\sigma'=\sigma^\delta=(1\ 3\ 8\ 2\ 4)$ and  $\tau'=\tau^\delta=(4\ 2\ 6\ 1)$, then $\sigma'$ and $\tau' $ are two permutation in $T_{5}^{10}$ such that  $\sigma' \tau'=\pi$. 
\end{exm}
For any $\pi\in S_n$, the centralizer $C_{S_n}(\pi)$ acts by conjugation on the set $\Omega_{n,r}(\pi)$ as follows:
$$(\sigma,\tau)\cdot \delta:=(\sigma^\delta,\tau^\delta)$$
for all $\delta \in C_{S_n}(\pi)$ and all $(\sigma,\tau)\in \Omega_{n,r}(\pi)$: for it follows from 
 $\sigma \tau =\pi$ that $\sigma^\delta \tau^\delta =\pi^\delta=\pi$; and since $|M(\eta)|=|M(\eta^\theta))|$ for all $\eta,\theta \in S_n$, we have $(\sigma^\delta,\tau^\delta)\in \Omega_{n,r}(\pi)$.

\begin{lem} \label{lem2}
Let $r\in [n]$ such that $2r\leq n$ and  $\pi \in S_{2r}$. Then   the number of orbits of the action by conjugation of $C_{S_n}(\pi)$ on $\Omega_{n,r}(\pi)$ is equal to the one of the action of $C_{S_{2r}}(\pi)$ by conjugation on $\Omega_{2r,r}(\pi)$. 
\end{lem}
\begin{proof}
	It follows from Lemma~\ref{lem1} that every orbit of the conjugation action of $C_{S_n}(\pi)$ on $\Omega_{n,r}(\pi)$ contains at least one element of $\Omega_{2r,r}(\pi)$. Hence, the number of orbits of the conjugation action of $C_{S_{2r}}(\pi)$ on $\Omega_{2r,r}(\pi)$ is at least as large as the number of orbits of the conjugation action of $C_{S_n}(\pi)$ on $\Omega_{n,r}(\pi)$.
Now suppose that $(\sigma,\tau)$ and $(\sigma',\tau')$ are two distinct elements of $\Omega_{2r,r}(\pi)$ that belong to the same orbit under the action of $C_{S_n}(\pi)$ on $\Omega_{n,r}(\pi)$. Then there exists $\lambda \in C_{S_n}(\pi)$ such that $(\sigma^\lambda,\tau^\lambda) = (\sigma',\tau')$.
Let $\Theta := [2r] \setminus (M(\sigma) \cup M(\tau))$ and $\Theta' := [2r] \setminus (M(\sigma') \cup M(\tau'))$. Since $|\Theta| = |\Theta'|$, there exists a bijection $\rho : \Theta \to \Theta'$. Define the permutation $\delta \in S_{2r}$ by setting $i^\delta := i^\rho$ for $i \in \Theta$, and $i^\delta := i^\lambda$ for $i \in [2r] \setminus \Theta$. Clearly, $\pi^\delta=(\sigma\tau)^\delta=\sigma^\delta\tau^\delta=\sigma'\tau'=\pi$ and therefore $\delta \in C_{S_{2r}}(\pi)$, which completes the proof.
\end{proof}
\begin{lem}\label{lem3}
Let $r\in [n]$ be such that $2r\leq n$, and let $\pi \in S_{2r}$. Suppose that  $\{(\sigma_1,\tau_1),\dots,(\sigma_s,\tau_s)\}$ is a complete set of  representatives of orbits of the action by conjugation of $C_{S_{2r}}(\pi)$ on $\Omega_{2r,r}(\pi)$. Then 
$$|\Omega_{n,r}(\pi)|=\sum_{i=1}^s \frac{|C_{S_n}(\pi)|}{|C_{S_n}(\pi) \cap C_{S_n}(\sigma_i)|}.$$
\end{lem}
\begin{proof}
	It follows from the fact that the underlying set of an action  is partitioned by the orbits and the size of the orbit of $(\sigma,\tau)\in \Omega_{n,r}(\pi)$ is equal to $|C_{S_n}(\pi):{\rm Stab}_{C_{S_n}(\pi)}((\sigma,\tau))|$. It is easy to see that ${\rm Stab}_{C_{S_n}(\pi)}((\sigma,\tau))=C_{S_n}(\sigma) \cap C_{S_n}(\tau)$. Since $\sigma \tau=\pi$, it follows that $C_{S_n}(\sigma) \cap C_{S_n}(\tau)=C_{S_n}(\sigma) \cap C_{S_n}(\pi)$. Now Lemma \ref{lem2} completes the proof. 
\end{proof}
We denote by $C_{S_A}(\sigma,\pi)$ the intersection of two centralizers $C_{S_A}(\sigma)$ and $C_{S_A}(\pi)$ that is $C_{S_A}(\sigma) \cap C_{S_A}(\pi)$.
\begin{lem}\label{lem4}
	For any two permutations $\sigma$ and $\pi$ in $S_n$, $$C_{S_n}(\sigma, \pi)=C_{S_{M(\sigma)\cup M(\pi)}}(\sigma,\pi) \times 
		S_{[n] \setminus \big( M(\sigma) \cup M(\pi) \big)}.$$
\end{lem}
\begin{proof}
It suffices to prove that for each $\tau \in C_{S_n}(\sigma, \pi)$ and each $i \in M(\sigma) \cup M(\pi)$, we have $i^\tau \in M(\sigma) \cup M(\pi)$. 
Suppose  that $i^\tau = j \in [n] \setminus \big( M(\sigma) \cup M(\pi) \big)$. Without loss of generality, we may assume that  $i\in M(\sigma)$.
Since $\tau \in C_{S_n}(\sigma)$, we have $i^{\sigma\tau} = i^{\tau\sigma}$. 
Suppose that $i^\sigma = k \in M(\sigma) \setminus \{i\}$. 
Then we would have $i^{\tau} = k^\tau = j$, which is a contradiction. 
This completes the proof.
\end{proof}
\begin{lem}\label{lem5}
Let $r\in [n]$ such that $2r\leq n$ and $\pi \in S_{2r}$. Suppose that  $\{(\sigma_1,\tau_1),\dots,$ $(\sigma_s,\tau_s)\}$ is a complete set of  representatives of orbits of the action by conjugation of $C_{S_{2r}}(\pi)$ on $\Omega_{2r,r}(\pi)$. Then 
\begin{equation}\label{rel1}
|\Omega_{n,r}(\pi)|=\sum_{i=1}^s \frac{|C_{S_{M(\pi)}}(\pi)|}{|C_{S_{M(\pi)\cup M(\sigma_i)}}(\pi,\sigma_i)|}\cdot \frac{\big(n-|M(\pi)|\big)!}{\big(n-|M(\pi)\cup M(\sigma_i)|\big)!}.
\end{equation}
\end{lem}
\begin{proof}
Since $M(I_n)=\varnothing$ and 	 by Lemma \ref{lem4}, $$C_{S_n}(\pi)=C_{S_n}(\pi, I_n)=C_{S_{M(\sigma) \cup \varnothing}}(\pi) \times S_{[n]\setminus \big(M(\sigma) \cup \varnothing \big)}=C_{S_{M(\sigma)}}(\pi) \times S_{[n]\setminus M(\sigma)}.$$ 
	Hence the proof follows from Lemmas \ref{lem3} and \ref{lem4}. 
\end{proof}

In view of the above results, if $n \geq 2r$, the value of $N(n,r)$ can be obtained from Remark \ref{rem} together with the relation (\ref{rel1}).  
If $r \leq n \leq 2r-1$, then $N(n,r)$ is directly obtained from 
$
   N(n,r) = \max \bigl\{\, |\Omega_{n,r}(\sigma_i)| : 1 \leq i \leq \ell \,\bigr\},
$
where $\sigma_1, \ldots, \sigma_\ell$ denote a complete set of representatives of the conjugacy classes of $S_n$.  
Based on these observations, we present Algorithm~1 for computing $N(n,r)$ for a given $r\leq n$.

\begin{algorithm}
		\caption{Computation of $N(n,r)$ for a given $r\leq n$}\label{algo}
		\begin{algorithmic}[1]

			\Input  $r\geq 2$ 
			\Output $N(n,r)$
			\State compute $T_r^{m}$ for each $m\in [2r]\setminus [r-1]$
			\State compute a complete set $R_m$ of representatives of non-trivial conjugacy classes of elements of $S_{m}$ for each $m\in [2r]\setminus [r-1]$
			\State compute $\Omega_{m,r}(\pi)$ for each $\pi\in R_m$ and all $m\in [2r]\setminus [r-1]$
			\State compute the set $\mathcal{O}_r$  consisting of the orbits $O_{\pi,r}$ of the action of $C_{S_{2r}}(\pi)$ on $\Omega_{2r,r}(\pi)$ for each $\pi\in R_{2r}$
			\State compute a set $\mathcal{RO}_r$ consisting of sets $RO_{\pi,r}$ of representatives of $O_{\pi,r}$ for each $\pi \in R_{2r}$
			\State compute the set of polynomials (in $n$) $|\Omega_{n,r}(\pi)|=\sum_{(\sigma,\tau) \in RO_{\pi,r}} \frac{|C_{S_{M(\pi)}}(\pi)|}{|C_{S_{M(\pi)\cup M(\sigma)}}(\pi,\sigma)|}\cdot \frac{\big(n-|M(\pi)|\big)!}{\big(n-|M(\pi)\cup M(\sigma)|\big)!}$ for each $RO_{\pi,r}\in \mathcal{RO}_r$.
			\State  Now $N(n,r)=\begin{cases}
			 \max\{|\Omega_{n,r}(\pi)| \; : \; \pi \in R_{2r}\} & {\rm if} \; n\geq 2r \\     \max\{|\Omega_{n,r}(\pi)| \; : \; \pi \in R_n\} & {\rm if} \; r\leq n\leq 2r-1. \end{cases}$
			\end{algorithmic}

	\end{algorithm}
The \textsf{GAP}~\cite{GAP} program corresponding to Algorithm~1 is given in  Appendix \ref{appen}. Also, in Table~\ref{T1}, we present the results obtained from Algorithm~1 
for $r=5,6,7$ and $r \leq n \leq 2r-1$, where each row corresponds to a conjugacy 
class of $S_n$ represented by its cycle type, together with the associated value 
of $|\Omega_{n,r}(\sigma)|$ for a representative $\sigma$ of that class. 
It is worth noting that $N(n,n)$ is clearly equal to $n!$, and this value is also included in the table for completeness. We are now ready to prove Theorem \ref{th5}.
\vspace*{.5cm}

\noindent\textbf{Proof of Theorem \ref{th5}}:
We consider \(r = 5, 6,\) or \(7\). Based on the values of \(N(n,r)\) obtained from Algorithm~1, which are listed in Table \ref{T1}, it can be verified that for \(n = r, r+1, \ldots, 2r - 1\), the corresponding values of \(N(n,r)\) satisfy the releations in parts (i), (ii), and (iii) for \(r = 5, 6, 7\), respectively.

Now assume that \(n \geq 2r\). In this case, the number of conjugacy classes in \(S_n\) whose elements move at most \(2r\) points equals the number of conjugacy classes in \(S_{2r}\), which is exactly 41, 77, and 135 for \(r=5, 6,\) and \(7\), respectively. 
By running Algorithm~1 in \textsf{GAP}~\cite{GAP}, we obtain a list of 41, 77, and 135 polynomial functions in \(n\) for \(r=5,6,7\), respectively, where each function represents the value of \(|\Omega_{n,r}(\sigma)|\) for some conjugacy class representative \(\sigma\). The \textsf{GAP} code and the resulting functions are provided in Appendix \ref{appen}.

To identify which of these functions yields the maximum value, we import the list into \textsf{MATLAB}~\cite{MATLAB} and use the ``isAlways'' command to perform symbolic comparisons. These comparisons show that, for \(n\geq 2r\), the function corresponding to transpositions provides the largest value.
Therefore, the formulas in parts (i), (ii), and (iii) hold for all \(n \geq r\), completing the proof.
\qed
\begin{table}
%\begin{sidewaystable}
	\centering
	\caption{\small{Computation of $N(n,r)$ using Algorithm 1 for  $r\in\{5,6,7\}$ and $n\in[2r]-[r-1]$}}\label{T1}
	\label{tab:part1}
	\fontsize{5.5}{6.5}\selectfont
	\setlength{\tabcolsep}{0.8pt}
	\begin{tabular}{@{}l GWGWGWGWGWGWGWGWGW@{}}
		\toprule

		\multirow{2}{*}{\footnotesize{Cycle Type}} & \multicolumn{5}{c}{$r=5$} & \multicolumn{6}{c}{$r=6$} & \multicolumn{7}{c}{$r=7$} \\
		\cmidrule(lr){2-6} \cmidrule(lr){7-12} \cmidrule(lr){13-19}
		& \rotatebox{90}{(5,5)} & \rotatebox{90}{(6,5)} & \rotatebox{90}{(7,5)} & \rotatebox{90}{(8,5)} & \rotatebox{90}{(9,5)} & \rotatebox{90}{(6,6)} & \rotatebox{90}{(7,6)} & \rotatebox{90}{(8,6)} & \rotatebox{90}{(9,6)} & \rotatebox{90}{(10,6)} & \rotatebox{90}{(11,6)} & \rotatebox{90}{(7,7)} & \rotatebox{90}{(8,7)} & \rotatebox{90}{(9,7)} & \rotatebox{90}{(10,7)} & \rotatebox{90}{(11,7)} & \rotatebox{90}{(12,7)} & \rotatebox{90}{(13,7)} \\
		\midrule

		$2$ & 120 & 358 & 802 & 1516 & 2564 & 720 & 2612 & 6946 & 15234 & 29350 & 51530 & 5040 & 21514 & 66222 & 165318 & 357458 & 696228 & 1252572 \\
		$3$ & 120 & 327 & 678 & 1206 & 1944 & 720 & 2409 & 5931 & 12189 & 22245 & 37320 & 5040 & 20013 & 57216 & 133797 & 273402 & 507102 & 874320 \\
		$4$ & 120 & 304 & 576 & 936 & 1384 & 720 & 2248 & 5128 & 9784 & 16640 & 26120 & 5040 & 18768 & 50064 & 109560 & 210360 & 368040 & 600648 \\
		$5$ & 120 & 285 & 465 & 660 & 870 & 720 & 2115 & 4360 & 7510 & 11620 & 16745 & 5040 & 17715 & 43645 & 87995 & 156195 & 253940 & 387190 \\
		$6$ & 120 & 270 & 354 & 438 & 522 & 720 & 2004 & 3594 & 5490 & 7692 & 10200 & 5040 & 16818 & 37518 & 68604 & 111540 & 167790 & 238818 \\
		$7$ &  & 238 & 245 & 252 & 252 & 720 & 1911 & 2821 & 3752 & 4704 & 5677 & 5040 & 16051 & 31472 & 51380 & 75852 & 104965 & 138796 \\
		$8$ &  &  &  & 98 & 98 & 720 &  & 2032 & 2240 & 2448 & 2656 & 5040 & 15392 & 25376 & 36112 & 47600 & 59840 & 72832 \\
		$9$ &  &  &  &  & 2 & 720 &  &  & 909 & 918 & 927 & 5040 &  & 19179 & 22608 & 26064 & 29547 & 33057 \\
		$10$ &  &  &  &  &  & 720 &  &  &  & 170 & 170 & 5040 &  &  & 10570 & 10990 & 11410 & 11830 \\
		$11$ &  &  &  &  &  & 720 &  &  &  &  & 11 & 5040 &  &  &  & 2706 & 2717 & 2728 \\
		$12$ &  &  &  &  &  &  &  &  &  &  &  &  &  &  &  &  & 312 & 312 \\
		$13$ &  &  &  &  &  &  &  &  &  &  &  &  &  &  &  &  &  & 13 \\
		$2^2$ & 120 & 306 & 586 & 964 & 1444 & 720 & 2252 & 5162 & 9914 & 16990 & 26890 & 5040 & 18786 & 50242 & 110358 & 212838 & 374228 & 614004 \\
		$2^3$ &  & 270 & 342 & 414 & 486 & 720 & 2004 & 3582 & 5454 & 7620 & 10080 & 5040 & 16830 & 37518 & 68532 & 111300 & 167250 & 237810 \\
		$2^4$ &  &  &  & 102 & 102 & 720 &  & 2062 & 2302 & 2548 & 2800 & 5040 & 15406 & 25390 & 36340 & 48268 & 61186 & 75106 \\
		$2^5$ &  &  &  &  &  &  &  &  &  & 260 & 260 & 5040 &  &  & 10660 & 11260 & 11860 & 12460 \\
		$3^2$ & 120 & 272 & 356 & 442 & 530 & 720 & 2004 & 3602 & 5518 & 7756 & 10320 & 5040 & 16820 & 37540 & 68698 & 111810 & 168410 & 240050 \\
		$3^3$ &  &  &  &  & 18 &  &  &  & 948 & 966 & 984 & 5040 &  & 19200 & 22656 & 26166 & 29730 & 33348 \\
		$4^2$ &  &  &  & 98 & 98 & 720 &  & 2050 & 2258 & 2468 & 2680 & 5040 & 15394 & 25378 & 36132 & 47660 & 59966 & 73054 \\
		$5^2$ &  &  &  &  &  &  &  &  &  &  &  &  &  &  &  &  &  &  \\
		$3,2$ & 120 & 286 & 466 & 660 & 868 & 720 & 2116 & 4366 & 7528 & 11660 & 16820 & 5040 & 17722 & 43690 & 88150 & 156590 & 254780 & 388772 \\
		$4,2$ & 120 &  & 350 & 430 & 510 & 720 &  & 2026 & 2226 & 2426 & 2626 & 5040 & 16822 & 37518 & 68580 & 111460 & 167610 & 238482 \\
		$5,2$ &  &  & 242 & 252 & 262 & 720 & 1912 & 2822 & 3762 & 4732 & 5732 & 5040 & 16054 & 31470 & 51398 & 75948 & 105230 & 139354 \\
		$6,2$ &  &  &  & 72 & 72 & 720 &  & 2026 & 2242 & 2458 & 2674 & 5040 & 15394 & 25378 & 36154 & 47722 & 60082 & 73234 \\
		$7,2$ &  &  &  &  & 14 & 720 &  &  & 924 & 938 & 952 & 5040 &  & 19182 & 22626 & 26112 & 29640 & 33210 \\
		$8,2$ &  &  &  &  &  &  &  &  &  & 176 & 176 & 5040 &  &  & 10576 & 11056 & 11536 & 12016 \\
		$9,2$ &  &  &  &  &  &  &  &  &  &  & 18 & 5040 &  &  &  & 2790 & 2808 & 2826 \\
		$4,3$ & 120 &  & 230 & 238 & 246 & 720 & 1910 & 2824 & 3756 & 4706 & 5674 & 5040 & 16052 & 31474 & 51382 & 75852 & 104960 & 138782 \\
		$5,3$ &  &  &  & 82 & 82 & 720 &  & 2032 & 2232 & 2432 & 2632 & 5040 & 15392 & 25368 & 36086 & 47546 & 59748 & 72692 \\
		$6,3$ &  &  &  &  & 12 & 720 &  &  & 922 & 934 & 946 & 5040 &  & 19174 & 22636 & 26144 & 29698 & 33298 \\
		$7,3$ &  &  &  &  &  &  &  &  &  & 154 & 154 & 5040 &  &  & 10544 & 10964 & 11384 & 11804 \\
		$8,3$ &  &  &  &  &  &  &  &  &  &  & 16 & 5040 &  &  &  & 2736 & 2752 & 2768 \\
		$5,4$ &  &  &  &  & 2 &  &  &  & 882 & 892 & 902 & 5040 &  & 19164 & 22602 & 26062 & 29544 & 33048 \\
		$6,4$ &  &  &  &  &  &  &  &  &  & 176 & 176 & 5040 &  &  & 10576 & 10984 & 11392 & 11800 \\
		$7,4$ &  &  &  &  &  &  &  &  &  &  & 2 & 5040 &  &  &  & 10964 & 10964 & 10964 \\
		$3,2^2$ &  & 270 & 238 & 252 & 266 & 720 & 1912 & 2826 & 3776 & 4762 & 5784 & 5040 & 16058 & 31470 & 51418 & 76044 & 105490 & 139898 \\
		$4,2^2$ &  &  &  & 82 & 82 & 720 & 2038 &  & 2262 & 2488 & 2716 & 5040 & 15398 & 25382 & 36216 & 47904 & 60450 & 73858 \\
		$5,2^2$ &  &  &  &  & 12 & 720 &  & 2032 &  & 932 & 952 & 5040 &  & 19170 & 22638 & 26158 & 29730 & 33354 \\
		$6,2^2$ &  &  &  &  &  &  &  &  &  & 188 & 188 & 5040 &  &  & 10588 & 11116 & 11644 & 12172 \\
		$7,2^2$ &  &  &  &  &  &  &  &  &  &  & 28 & 5040 &  &  &  & 10932 & 10932 & 10932 \\
		$4,3,2$ &  &  &  &  & 10 & 720 &  &  & 910 & 928 & 946 & 5040 &  & 19150 & 22660 & 26236 & 29878 & 33586 \\
		$5,3,2$ &  &  &  &  &  &  &  &  &  & 192 & 192 & 5040 &  &  & 10622 & 11102 & 11584 & 12068 \\
		$6,3,2$ &  &  &  &  &  &  &  &  &  &  & 14 & 5040 &  &  & 10576 & 10984 & 11384 & 11804 \\
		$3^2,2$ & 120 &  &  & 76 & 76 & 720 & 15394 &  &  & 11044 &  & 5040 & 15394 & 25362 & 36102 & 47614 & 59698 & 72954 \\
		$4,3^2$ &  &  &  &  &  &  &  &  &  & 220 & 220 & 5040 &  &  & 10680 & 11208 & 11740 & 12276 \\
		$5,4,2$ &  &  &  &  &  &  &  &  &  & 202 & 202 & 5040 &  &  & 10642 & 11062 & 11484 & 11908 \\
		$3,2^3$ &  &  &  &  & 6 & 720 &  &  & 886 & 916 & 946 & 5040 &  & 19150 & 22600 & 26144 & 29698 & 33298 \\
		$4,2^3$ &  &  &  &  &  &  &  &  &  & 212 & 212 & 5040 &  &  & 10612 & 11164 & 11716 & 12268 \\
		$5,2^3$ &  &  &  &  &  &  &  &  &  &  & 32 & 5040 &  &  &  & 2972 & 3014 & 3056 \\
		$3^2,3$ &  &  &  &  &  &  &  &  & 948 &  &  & 5040 &  & 19186 & 22624 & 26098 & 29608 & 33154 \\
		$4,3,2^2$ &  &  &  &  &  &  &  &  &  & 144 & 144 & 5040 &  &  & 10524 & 10524 & 11340 & 11748 \\
		$5,3,2^2$ &  &  &  &  &  &  &  &  &  &  & 12 & 5040 &  &  &  & 2702 & 2724 & 2746 \\
		$3^2,2^2$ &  &  &  &  &  &  &  &  &  & 202 & 202 & 5040 &  &  & 10642 & 11116 & 11584 & 12068 \\
		$4^2,2$ &  &  &  &  &  &  &  &  &  & 188 & 188 & 5040 &  &  & 10588 & 11044 & 11500 & 11956 \\
		\bottomrule

		$N(n,r)$ & $120$ & 358 & 802 & 1516 & 2564 & 720 & 2612 & 6946 & 15234 & 29350 & 51530 & 5040 & 21514 & 66222 & 165318 & 357458 & 696228 & 1252572 \\
		\bottomrule

	\end{tabular}
\end{table}
\begin{rem}
Note that, according to the proof of Theorem \ref{th5}, for $r \in \{5,6,7\}$ we have 
$N(n,r) = |\Omega_{n,r}(\tau)|$, where $\tau$ is a transposition. Therefore, in view of 
Remark \ref{rem}, Conjecture \ref{con} holds for these values of $r$.
It is worth noting that Algorithm~1 for computing $N(n,r)$ for $r \geq 8$ was not 
feasible in GAP \cite{GAP}, since computing the values of $T_n^r$ and the orbits of 
the action specified in the algorithm became infeasible for such $r$. In the next section, 
we present a method that enables the computation of $N(n,r)$ for $r \geq 8$, 
even up to $r = 14$, for small values of $n$.
\end{rem}
\section{Computation of $N(n, r)$ Based on Irreducible Characters of $S_n$}
In this section, we present a formula for computing $N(n,r)$, which is based on the irreducible characters of the symmetric group $S_n$. We then provide an algorithm implementing this formula.  
One of the main advantages of this approach is that it allows the computation of $N(n,r)$ for larger values of the parameters. For instance, quantities such as $N(43,8)$ or $N(24,14)$ can be determined using this method. At the end of this section, we present these computations in Table \ref{table2}.  

We begin by reviewing the necessary notation.  
Let \(\operatorname{Irr}(S_n)\) denote the set of all complex irreducible characters of \(S_n\), and let \(\mathcal{C}_1, \dots, \mathcal{C}_m\) denote the conjugacy classes of \(S_n\). For each  \(i \in [m]\) and each character \(\chi \in \operatorname{Irr}(S_n)\), we define \(\chi_i := \chi(\sigma)\), where \(\sigma \in \mathcal{C}_i\). Since characters are constant on conjugacy classes, the value \(\chi_i\) is well-defined.
\begin{lem}\label{31}
Let  $n\geq 3$ and $2\leq r\leq n$ be integers   and let $\sigma\in S_n$. Let \(\mathcal{C}_1, \dots, \mathcal{C}_m\) denote the conjugacy classes of \(S_n\), where \(\mathcal{C}_1\) is the class containing the identity element. If $\sigma \in    \mathcal{C}_s$, then 
\[
|\Omega_{n, r}(\sigma)| =  \sum_{i,j \in \Delta_r^n } \frac{|\mathcal{C}_i||\mathcal{C}_j|}{n!} \sum_{\chi \in \operatorname{Irr}(S_n)} \frac{\chi_i \chi_j \chi_s}{\chi_1},
\]
where $\Delta^n_{r}$ are the sets of indices corresponding to classes whose elements move at most $r$   points.
\end{lem}
\begin{proof}
According to the definition of \(\Omega_{n, r}(\sigma)\), we have
\[
|\Omega_{n, r}(\sigma)| = \left| \cup_{i,j \in \Delta_r^n} \{(\tau, \rho) \in \mathcal{C}_i \times \mathcal{C}_j \mid \tau \rho = \sigma\} \right|.
\]

Let \(i, j, s \in [m]\). For each \(\sigma \in \mathcal{C}_s\), the number of pairs \((x, y) \in \mathcal{C}_i \times \mathcal{C}_j\) such that \(xy = \sigma\) is clearly constant; we denote this number by \(\mathcal{C}_{i,j}^{s}\). Hence,
$
|\Omega_{n, r}(\sigma)| = \sum_{i,j \in \Delta_r^n} \mathcal{C}_{i,j}^{s}.
$
Moreover, due to the fact that for every $\tau \in S_n$, $\tau$ and $\tau^{-1}$ are into the same conjugate class and in view of \cite[Lemma 2.15 and problem 3.9]{isaacs},  for all \(i, j, s \in [m]\), we have
\[
\mathcal{C}_{i,j}^{s} = \frac{|\mathcal{C}_i||\mathcal{C}_j|}{n!} \sum_{\chi \in \operatorname{Irr}(S_n)} \frac{\chi_i \chi_j \chi_s}{\chi_1}.
\]
This completes the proof.
\end{proof}
\begin{thm}\label{thchar}
Let  $n\geq 3$ and $2\leq r\leq n$ be integers   and let   \(\mathcal{C}_1, \dots, \mathcal{C}_m\) denote the conjugacy classes of \(S_n\), where \(\mathcal{C}_1\) is the class containing the identity element. Then
\begin{equation}\label{formul}
N(n, r) = \max_{s \in \Delta_{2r}^n} \sum_{i,j \in \Delta_r^n } \frac{|\mathcal{C}_i||\mathcal{C}_j|}{n!} \sum_{\chi \in \operatorname{Irr}(S_n)} \frac{\chi_i \chi_j \chi_s}{\chi_1},
\end{equation}
where $\Delta^n_{r}$ and $\Delta^n_{2r}$ are the sets of indices corresponding to classes whose elements move at most $r$ and $2r$  points, respectively.
\end{thm}
\begin{proof}
The result follows from Lemma \ref{31} and Remark \ref{rem}.
\end{proof}
Algorithm~2 is presented for computing $N(n,r)$ for given values of $n$ and $r$ as expressed in relation (\ref{formul}).
As specified in the algorithm, the output of \texttt{ComputeClass} \texttt{OmegaList} is a list of ordered pairs, where the first element of each pair is a representative of a conjugacy class that moves at most $2r$ elements, and the second element is the value of $\Omega_{n,r}$ for that permutation. The function \texttt{computeN} returns a list of pairs from the list produced by \texttt{computelistomega} whose second components are equal to the maximum value among all second components in the list.  
In fact, the second component of each element in the output corresponds to $N(n,r)$, while the first component identifies the conjugacy class whose representative achieves this value of $\Omega_{n,r}$.

By running Algorithm 2 in \textsf{GAP} \cite{GAP}, the values of $N(n,r)$ were computed as summarized in Table \ref{table2}. 
In most cases the computations were performed quickly, and even in the most time-consuming instances the running time did not exceed three days; entries in Table~\ref{table2} marked with ``--'' indicate cases where the computation required more time.  
Because of the large number of conjugacy classes whose representatives move at most $2r$ points, we did not report the values of $\Omega_{n,r}(\sigma)$ for each $\sigma$. 
It is worth noting that all results in Table \ref{table2} confirm Conjecture \ref{con}; 
in other words, in every case, $N(n,r)$ coincides with $I(n,2,r)=|\Omega_{n,r}(\tau)|$, where $\tau$ is a transposition in $S_n$.
%\end{sidewaystable}
{\fontsize{5.5}{6.5}\selectfont
	\setlength{\tabcolsep}{0.8pt}
\begin{longtable}{|>{\columncolor{lightgray}}c|*{7}{c|}}
\caption{\small{$N(n,r)$ obtained from Algorithm~2 for $8 \leq r \leq 14$}}\label{table2}\\
\hline
\rowcolor{lightgray}
$n \diagdown r$ & $8$ & $9$ & $10$ & $11$ & $12$ & $13$ & $14$ \\
\hline
\endfirsthead
\hline
%\rowcolor{lightgray}
%$n \diagdown r$ & $8$ & $9$ & $10$ & $11$ & $12$ & $13$ & $14$ \\
%\hline
%\endhead
8  & 40320 &   &   &   &  &   &   \\
9  & 197864 & 362880 &   &   &   &  &  \\
10 & 691886 & 2012014 & 3628800 &   &   &   &   \\
11 & 1937162 & 7877738 & 22428812 & 39916800 &   &   &   \\
12 & 4645488 & 24447408 & 97202778 & 272082658 & 479001600 &   &   \\
13 & 9940944 & 64396224 & 331165914 & 1293005254 & 3569113616 & 6227020800 &  \\
14 & 19493964 & 150186636 & 950495706 & 4797853066 & 18454503238 & 50349389446 & 87178291200 \\
15 & 35674212 & 318841668 & 2399645250 & 14903556670 & 74082374082 & 281399134434 & 760174857236 \\
16 & 61722264 & 628057176 & 5483265534 & 40494217510 & 247620078452 & 1215098293428 & 4566528353042 \\
17 & 101940096 & 1163818056 & 11567835966 & 99095215906 & 720472798732 & 4348516104892 & 21105666402962 \\
18 & 161900378 & 2049683418 & 22857719238 & 222920301958 & 1879927189494 &13489665769206 & 80518266961486 \\
19 & 248674574 & 3457905742 & 42761973402 & 467894961682 & 4492054545698 & 37386313854882 & 265283557908634 \\
20 & 371079848 & 5622549032 & 76369870820 & 926635847380 & 9980994911416 & 94566233135032 & 778257965296288 \\
21 & 539944776 & 8854770984 & 131054690436 & 1746560045900 & 20861318069976 &  221751258851064 &  2077485960431616 \\
22 & 768393864 &  13560434184 & 217226966604 & 3154509431084 &41384025479236 &  487806116103876& 5127142870055256 \\
23 & 1072150872 & 20260211352 & 349259994492 & 5489504307332 & 78473125853804 & 1015959516165548 & 11841078608718768 \\
24 & 1469860944 & 29612349648 & 546612008868 & 9245486999828 & 143048793168360 & 2018021573791848 & 25833283577408932 \\
25 & 1983431544 & 42438259056 & 835171069860 & 15126179983580 & 251855106281752 & 3845552935810104 & -- \\
26 & 2638392198 & 59751089862 & 1248850306068 & 24114464568020 & 429935481079172 & -- & -- \\
27 & 3464273042 & 82787464242 & 1831462782192 & 37558985066492 & 713927345918412 & -- & -- \\
28 & 4495002176 & 113042526976 & 2638906875126 & 57280999780526 &1156379586164896 & -- & -- \\
29 & 5769321824 & 152308480304 & 3741694659254 &85704834017354 & -- & -- & -- \\
30 & 7331223300 & 202716767940 & 5227857418470 & 126015641735670 & -- & -- & -- \\
31 & 9230400780 & 266784073260 & 7206264019230 & 182348551279170 & -- & -- & -- \\
32 & 11522723880 & 347462296680 & 9810389495730 & 260013657009930 & -- & -- & -- \\
33 & 14270729040 & 448192677240 & 13202572815090 & 365761722494190 & -- & -- & -- \\
34 & 17544129714 & 572964223410 & 17578804407210 &508095882221610 & -- & -- & -- \\
35 & 21420345366 & 726376618134 & 23174085660750 & 697635067655550 & -- & -- & -- \\
36 & 25985049272 & 913707763128 & 30268404203472 &947535339716400 & -- & -- & -- \\
37 & 31332735128 & 1140986127448 & 39193370401968 & 1273975783592448 & -- & -- & -- \\
38 & 37567302464 & 1415068065344 & 50339562132584 & -- & -- & -- & -- \\
39 & 44802660864 & 1743720268416 & 64164626492136 & -- & -- & -- & -- \\
40 & 53163352992 & 2135707517088 & 81202188733800 & -- & -- & -- & -- \\
41 & 62785196424 &--  & -- & -- & -- & -- & -- \\
42 & 73815944286 & -- & -- & -- & -- & -- & -- \\
43 & 86415964698 & -- & -- & -- & -- & -- & -- \\
\hline

\end{longtable}}
\begin{algorithm}[H]
\caption{Compute \(N(n, r)\) as obtained from Equation (\ref{formul})}
\begin{algorithmic}[1]
\Input Integer numbers  \(n \) and  \(2\leq r \leq n\)
\Output A list of pairs $(\sigma, N(n,r))$, where each $\sigma$ is a representative of a conjugacy class whose $\Omega_{n,r}(\sigma)$ attains the value $N(n,r)$.

\Function{SumChar}{$G, i, j, \ell$}
  \State \(total \gets 0\)
  \ForAll{\(\chi \in \mathrm{Irr}(G)\)}
    \State \(total \gets total + \frac{\chi_i \cdot \chi_j \cdot \chi_\ell}{\chi_1}\)
  \EndFor
  \State \Return \(total\)
\EndFunction

\Function{ComputeClassOmegaList}{$n, r$}
  \State \(G \gets \) the symmetric group \(S_n\)
  \State \(C \gets\) a list of conjugacy classes of $G$
 \State \(\Delta_{r} \gets \{\,\ell \in [|C|]\, :\, C[\ell]\subseteq T_{r}^n\}\)
  \hfill\Comment{\textcolor{blue}{Indices of classes moving at most $r$ points}}
\State \(\Delta_{2r} \gets\)  a list of  the indices of the nontrivial conjugacy classes of $G$ that move at most $2r$ points.
  \State \(resultList \gets [\,]\) \hfill\Comment{\textcolor{blue}{Stores computed pairs \((\sigma_i, |\Omega_{n,r}(\sigma_i)|)\)}}
  \ForAll{\(\ell \in \Delta_{2r}\)}
    \State \(N_\ell \gets 0\) 
    \ForAll{\(i \in \Delta_r\)}
      \ForAll{\(j \in \Delta_r\)}
        \State \(w \gets \frac{|C[i]| \cdot |C[j]|}{|G|}\)
        \State \(s \gets \mathrm{SumChar}(G, i, j, \ell)\)
        \State \(N_\ell \gets N_\ell + w \cdot s\)
      \EndFor
    \EndFor
    \State Append \((\text{an representative of}\,\, C[\ell], N_\ell)\) to \(resultList\)
  \EndFor
  \State \Return \(resultList\)
\EndFunction
\Function{ComputeN}{$n, r$}
   \State \(data \gets \) ComputeClassOmegaList \(n,r)\)
  \State \Return The list of pairs in $data$ whose second components attain the maximum value among all second components of pairs in $data$.
\EndFunction
\end{algorithmic}
\end{algorithm}

\section{Appendix}\label{appen}
The following \textsf{GAP} \cite{GAP} code computes $N(n,5)$ using Algorithm 1 for $n \geq 10$:
{\scriptsize
\begin{verbatim}
u:=function(sigma,tau) 
return Union(MovedPoints(sigma),MovedPoints(tau));
end;;
su:=function(sigma,tau) 
return Size( u(sigma,tau));
end;;
csu:=function(sigma,tau)
return Size(Centralizer(SymmetricGroup(u(sigma,tau)),Group(sigma,tau)));
end;;
n:=Indeterminate(Rationals,"n");
p:=function(sigma,tau) 
local b,l,l1,i; 
l:=su(sigma,()); l1:= su(sigma,tau); 
if l=l1 then b:=1; fi; 
if l<l1 then b:=Product([l..l1-1],i->n-i); fi; 
return b; 
end;;
CU:=function(sigma,tau) 
return csu(sigma,())/csu(sigma,tau)*p(sigma,tau);
end;;
S10:=SymmetricGroup(10);
C10:=ConjugacyClasses(S10);
T5:=Union(Filtered(C10,x->Size(MovedPoints(Representative(x)))<=5));;
c:=List(C10,Representative);;
F:=List(c,a->[a,Filtered(T5,x->x^-1*a in T5)]);;
F2:=List(F,X->[X[1],Orbits(Centralizer(S10,X[1]),X[2])]);;
F3:=List(F2,X->[X[1],List(X[2],Representative)]);;
F4:=Difference(F3,[F3[1]]);;
List(F4,X->Sum(X[2],x->CU(X[1],x)));
\end{verbatim}}
In the following, we present the result of the above program in \textsf{GAP}:
\begin{lstlisting}
[32/3*n^3-89*n^2+739/3*n-220, 2/3*n^3+35*n^2-779/3*n+460, 72*n-162, 102,
0, 11/2*n^3-27*n^2+7/2*n+90,7*n^2+89*n-500, 14*n+140, 6, n^2+71*n-190, 
76, 4, 18, 44*n^2-300*n+520, 80*n-210, 82, 0, 8*n+174, 10, 0, 98, 0, 
15/2*n^2+165/2*n-480, 10*n+172, 12, 82, 2, 2, 2, 84*n-234, 72, 0, 12,
0, 7*n+189, 14, 0, 80, 0, 9, 0 ]
\end{lstlisting}
By running a similar program in \textsf{GAP} to compute $N(n,6)$ and $N(n,7)$, we obtained the following results, respectively.
\begin{lstlisting}
[ 181/12*n^4-401/2*n^3+11783/12*n^2-4153/2*n+1590, 
  3/4*n^4+329/6*n^3-2735/4*n^2+15769/6*n-3250, 
  147*n^2-627*n-810, 3*n^2+189*n+358, 260, 20, 
  53/8*n^4-193/4*n^3-197/8*n^2+3265/4*n-1455, 
  29/3*n^3+224*n^2-8231/3*n+7030, 18*n^2+644*n-3478, 
  30*n+616, 26, 2/3*n^3+143*n^2-1979/3*n-614, 
  200*n+426, 220, 2, 18*n+786, 6, 6, 
  212/3*n^3-808*n^2+9172/3*n-3800, 151*n^2-679*n-642, 
  n^2+207*n+318, 212, 8, 9*n^2+779*n-3984, 18*n+748, 
  18, 144, 2, n^2+191*n+458, 188, 4, 22, 0, 
  55/6*n^3+465/2*n^2-8375/3*n+7120, 15*n^2+685*n-3618, 
  20*n+732, 32, 200*n+432, 192, 0, 12, 10*n+792, 
  12, 0, 202, 0, 153*n^2-705*n-558, 216*n+298, 188, 2, 
  12*n+814, 14, 2, 176, 2, 2, 2, 21/2*n^2+1505/2*n-3871, 
  14*n+798, 28, 154, 0, 14, 0, 208*n+368, 176, 0, 16, 0, 
  9*n+828, 18, 0, 170, 0, 11, 0 ]
\end{lstlisting}
\begin{lstlisting}
[ 607/30*n^5-4675/12*n^4+8807/3*n^3-129041/12*n^2+190471/10*n-12978, 
  11/15*n^5+327/4*n^4-8917/6*n^3+37061/4*n^2-730247/30*n+22582, 238*n^3-1263*n^2-9487*n+51702, 
  2*n^3+429*n^2+2257*n-31130, 600*n+4660, 740, 0, 309/40*n^5-615/8*n^4-1031/8*n^3+31551/8*n^2-308017/20*n+18543, 
  47/4*n^4+2809/6*n^3-36467/4*n^2+302381/6*n-89810, 71/3*n^3+1629*n^2-52250/3*n+39018, 33*n^2+2883*n-9470, 66*n+2246, 
  20, 3/4*n^4+1345/6*n^3-4795/4*n^2-56365/6*n+50848, 386*n^2+3406*n-36558, 2*n^2+486*n+5620, 494, 12, 
  27*n^2+2943*n-9474, 36*n+2292, 42, 438, 0, 103*n^4-1730*n^3+10649*n^2-28222*n+26880, 242*n^3-1351*n^2-8851*n+50190, 
  2/3*n^3+407*n^2+8761/3*n-34354, 552*n+5092, 536, 0, 38/3*n^3+1901*n^2-58931/3*n+45052, 23*n^2+3025*n-9914, 
  38*n+2410, 32, 408*n+6444, 398, 4, 6, 2/3*n^3+367*n^2+10801/3*n-37238, 456*n+6028, 436, 0, 22*n+2546, 20, 4, 240, 
  0, 265/24*n^4+5825/12*n^3-222565/24*n^2+612325/12*n-90755, 55/3*n^3+1761*n^2-55498/3*n+41958, 26*n^2+2974*n-9702, 
  42*n+2510, 38, 371*n^2+3669*n-37704, n^2+459*n+5932, 414, 6, 22*n+2460, 34, 0, 11*n^2+3229*n-10788, 22*n+2510, 22, 
  294, 2, 26, n^2+399*n+6552, 384, 4, 26, 0, 244*n^3-1395*n^2-8533*n+49434, 396*n^2+3252*n-35966, 528*n+5308, 434, 0, 
  18*n^2+3096*n-10136, 24*n+2492, 38, 350, 0, 408*n+6496, 386, 0, 14, 0, 12*n+2510, 14, 0, 362, 0, 
  77/6*n^3+1897*n^2-117677/6*n+44975, 21*n^2+3045*n-9924, 28*n+2592, 44, 420*n+6344, 366, 2, 16, 14*n+2578, 16, 2, 
  324, 2, 2, 2, 376*n^2+3592*n-37408, 480*n+5776, 384, 0, 16*n+2560, 32, 0, 288, 0, 16, 0, 27/2*n^2+6345/2*n-10467, 
  18*n+2592, 36, 306, 0, 18, 0, 420*n+6370, 360, 0, 20, 0, 11*n+2585, 22, 0, 312, 0, 13, 0 ]
\end{lstlisting}

%\vspace*{-.2cm}

Here is a code written in \textsf{GAP} \cite{GAP} to compute $N(n,r)$ from Algorithm 2:
{\scriptsize
 \begin{verbatim}
sumchar:= function(G,i,j,s,NCju)
local A,irr,x;
irr:=Irr(G);;
A:=[];;
for x in [1..NCju] do
Add(A,(irr[x][i]*irr[x][j]*irr[x][s])/irr[x][1]);
od;
return(Sum(A)); 
end;
N:=function(n,r)
local G,Cju,NCju,Ir,I2r,H,k,list,c,l,i,j,cijl;
G:=SymmetricGroup(n);
Cju:=ConjugacyClasses(G);;
NCju:=NrConjugacyClasses(G);
Ir:=Filtered([1..NCju],t->NrMovedPoints(Representative(Cju[t]))<=r);
I2r:=Filtered([2..NCju],t->NrMovedPoints(Representative(Cju[t]))<=2*r);
resultlist:=[];
for l in I2r do
k:=0;
for i in Ir do
for j in Ir do
cijl:=((Size(Cju[i])*Size(Cju[j]))/Size(G))*sumchar(G,i,j,l,NCju);
k:=k+cijl;
od;
od;
Add(resultlist,[Representative(Cju[l]),k]);
od;
return(resultlist);
end;
ComputN:=functiona(n,r)
local data,max, maxdata;
data:=nr(n,r);;
max:=Maximum(List(data,i->i[2]));;
maxdata:=Filtered(data,i->i[2]=max);;
return(maxdata);
end;

\end{verbatim}}

\end{document}